\documentclass[12pt]{amsart}
\usepackage{amsmath, amsfonts, amssymb, amsthm,hyperref}
\hypersetup{hypertex=true,
	colorlinks=true,
	linkcolor=blue,
	anchorcolor=blue,
	citecolor=blue}
\usepackage{bm}
\allowdisplaybreaks[4]
\def\({\bg(}
\def\){\bg)}

\def\ord{{\rm ord}}

\def\Z{\mathbb Z}

\def\p{\mathfrak p}
\def\O{{\mathcal{O}}}
\def\Gal{{\rm Gal}}
\def\Tr{{\rm Tr}}
\def\Norm{{\rm Norm}}

\def\Ker{{\rm Ker}}
\def\v{{\bm v}}
\def\diag{{\rm diag}}
\def\pmod #1{\ ({\rm{mod}}\ #1)}
\def\mod #1{\ {\rm mod}\ #1}
\def\Ack{\medskip\noindent {\bf Acknowledgments}}

\theoremstyle{plain}
\newtheorem{theorem}{Theorem}[section]
\newtheorem{lemma}{Lemma}

\theoremstyle{definition}

\theoremstyle{remark}
\newtheorem{remark}{Remark}

\newcommand{\sign}[1]{\mathrm{sign}(#1)}
\makeatletter
\@namedef{subjclassname@2020}{%
	\textup{2020} Mathematics Subject Classification}
\makeatother
\vspace{4mm}

\begin{document}
	\medskip
	
	\title[On cyclotomic matrices involving Gauss sums over finite fields]
	{On cyclotomic matrices involving Gauss sums over finite fields}
	\author[H.-L. Wu, J. L, L.-Y. Wang and C. H. Yip]{Hai-Liang Wu, Jie Li, Li-Yuan Wang* and Chi Hoi Yip}
	
	\address {(Hai-Liang Wu) School of Science, Nanjing University of Posts and Telecommunications, Nanjing 210023, People's Republic of China}
	\email{\tt whl.math@smail.nju.edu.cn}
	
	\address {(Jie Li) School of Science, Nanjing University of Posts and Telecommunications, Nanjing 210023, People's Republic of China}
	\email{\tt lijiemath@163.com}
	
	\address {(Li-Yuan Wang) School of Physical and Mathematical Sciences, Nanjing Tech University, Nanjing 211816, People's Republic of China}
	\email{\tt wly@smail.nju.edu.cn}
	
	\address{(Chi Hoi Yip) Department of Mathematics, University of British Columbia, Vancouver  V6T 1Z2, Canada}
	\email{\tt kyleyip@math.ubc.ca}
	
	\keywords{Gauss Sums, Finite Fields, Cyclotomic Matrices, Determinants.
		\newline \indent 2020 {\it Mathematics Subject Classification}. Primary 11L05, 15A15; Secondary 11R18, 12E20.
		\newline \indent This work was supported by the Natural Science Foundation of China (Grant Nos. 12101321 and 12201291).	
        \newline \indent *Corresponding author.}
	
	\begin{abstract}
		Inspired by the works of L. Carlitz and Z.-W. Sun on cyclotomic matrices, in this paper, we investigate certain cyclotomic matrices involving Gauss sums over finite fields, which can be viewed as finite field analogues of certain matrices related to the Gamma function.
		
		 For example, let $q=p^n$ be an odd prime power with $p$ prime and $n\in\mathbb{Z}^+$. Let $\zeta_p=e^{2\pi{\bf i}/p}$ and let $\chi$ be a generator of the group of all mutiplicative characters of the finite field $\mathbb{F}_q$. For the Gauss sum 
		$$G_q(\chi^{r})=\sum_{x\in\mathbb{F}_q}\chi^{r}(x)\zeta_p^{{\rm Tr}_{\mathbb{F}_q/\mathbb{F}_p}(x)},$$
		  we prove that 
	$$\det \left[G_q(\chi^{2i+2j})\right]_{0\le i,j\le (q-3)/2}=(-1)^{\alpha_n}\left(\frac{q-1}{2}\right)^{\frac{q-1}{2}}2^{\frac{p^{n-1}-1}{2}},$$
		where 
		$$\alpha_n=
			\begin{cases}
				1            & \mbox{if}\ n\equiv 1\pmod 2,\\
				(p^2+7)/8    & \mbox{if}\ n\equiv 0\pmod 2.
			\end{cases}$$
	\end{abstract}
	\maketitle
	
	\section{Introduction}
	\setcounter{lemma}{0}
	\setcounter{theorem}{0}
	\setcounter{equation}{0}
	\setcounter{conjecture}{0}
	\setcounter{remark}{0}
	\setcounter{corollary}{0}
	
	\subsection{Notations}
	
Let $q=p^n$ be a prime power with $p$ prime and $n\in\mathbb{Z}^+$ and let $\mathbb{F}_q$ be the finite field of $q$ elements. Let $\mathbb{F}_q^{\times}$ be the cyclic group of all nonzero elements of $\mathbb{F}_q$, and let $\widehat{\mathbb{F}_q^{\times}}$ be the cyclic group of all multiplicative characters of $\mathbb{F}_q$. 
	
	Throughout this paper, for any $\psi\in\widehat{\mathbb{F}_q^{\times}}$, we define $\psi(0)=0$. Also, we use the symbol $\varepsilon$ to denote the trivial multiplicative character of $\mathbb{F}_q$, i.e., 
	\begin{equation*}
		\varepsilon(x)=\begin{cases}
			1 & \mbox{if}\ x\in\mathbb{F}_q^{\times},\\
			0 & \mbox{if}\ x=0.
		\end{cases}
	\end{equation*}
	Let $\zeta_p=e^{2\pi{\bf i}/p}$ and let $\Tr:=\Tr_{\mathbb{F}_q/\mathbb{F}_p}$ be the trace map from $\mathbb{F}_q$ to $\mathbb{F}_p$. For any $x\in\mathbb{F}_q$, we have 
    \begin{equation*}
    	 \Tr(x)=\sum_{j=0}^{n-1}x^{p^j}\in\mathbb{F}_p.
    \end{equation*} 
   Let $\zeta_{q-1}=e^{2\pi{\bf i}/(q-1)}$. For any $\psi\in\widehat{\mathbb{F}_q^{\times}}$, the Gauss sum $G_q(\psi)$ over $\mathbb{F}_q$ is defined by 
	\begin{equation}\label{Eq. definition of the Gauss sum}
		G_q(\psi)=\sum_{x\in\mathbb{F}_q}\psi(x)\zeta_p^{\Tr(x)}\in\mathbb{Q}(\zeta_{q-1},\zeta_p).
	\end{equation}

	\subsection{Backgroud and Motivations}
	
	Gauss sums have been extensively studied and have many significant applications in both number theory and combinatorics. For example, In 1805 Gauss first determined the explicit values of quadratic Gauss sums over $\mathbb{F}_p$, which states that 
	\begin{equation}\label{Eq. quadratic gauss sums over Fp}
		\sum_{x\in\mathbb{F}_p}\left(\frac{x}{p}\right)\zeta_p^x=\begin{cases}
			p^{1/2}          & \mbox{if}\ p\equiv 1\pmod 4,\\
			{\bf i}p^{1/2}   & \mbox{if}\ p\equiv 3\pmod 4,
		\end{cases}
	\end{equation}
    where $(\frac{\cdot}{p})$ is the Legendre symbol, i.e., the unique quadratic multiplicative character of $\mathbb{F}_p$. 
    
	Also, the Hasse-Davenport lifting formula says that for any positive integer $n$ we have 
	\begin{equation}\label{Eq. the HD lifting formula}
		G_{q}(\psi_{(n)})=(-1)^{n-1}G_p(\psi)^n,
	\end{equation}
	where $q=p^n$ and $\psi_{(n)}=\psi\circ\Norm_{\mathbb{F}_q/\mathbb{F}_p}$ is a multiplicative character of $\mathbb{F}_q$. 
	
	Using the Hasse-Davenport lifting formula (\ref{Eq. the HD lifting formula}), one can generalize (\ref{Eq. quadratic gauss sums over Fp}) to $\mathbb{F}_{q}$. In fact, when $p$ is an odd prime, let $\phi$ be the unique quadratic multiplicative character of $\mathbb{F}_q$, i.e., 
	\begin{equation*}
		\phi(x)=
		\begin{cases}
			0  & \mbox{if}\ x=0,\\
			1  & \mbox{if}\ x\ \text{is a nonzero square},\\
			-1 & \mbox{otherwise}.
		\end{cases}
	\end{equation*}
	Then it is known that 
	\begin{equation*}
		G_q(\phi)=\sum_{x\in\mathbb{F}_q}\phi(x)\zeta_p^{\Tr(x)}=
		\begin{cases}
			(-1)^{n-1}q^{1/2}           & \mbox{if}\ p\equiv 1\pmod 4,\\
			(-1)^{n-1}{\bf i}^nq^{1/2}  & \mbox{if}\ p\equiv 3\pmod 4.		
		\end{cases}
	\end{equation*}
		
Now we turn to Carlitz's cyclotomic matrices. Carlitz first studied the arithmetic properties of cyclotomic matrices involving multiplicative characters of $\mathbb{F}_p$, where $p$ is an odd prime. Let $\psi\in\widehat{\mathbb{F}_p^{\times}}$ be a nontrivial character with $\ord(\psi)=f$, where 
	\begin{equation*}
		\ord(\psi)=\min\{r\in\Z^+:\ \psi^r=\varepsilon\}.
	\end{equation*}
	Carlitz \cite[Theorem 5]{Carlitz} considered the cyclotomic matrix 
	\begin{equation}\label{Eq. definition of Carlitz's matrix C(psi)}
		C_p(\psi):=\left[\psi(i+j)\right]_{1\le i,j\le p-1},
	\end{equation} 
	and determined the explicit value  of $\det C_p(\psi)$, which states that 
	\begin{equation}\label{Eq. determinant of Carlitz's cyclotomic matrix}
		\det C_p(\psi)=
		\begin{cases}
			(-1)^{(p-1)/(2f)}G_p(\psi)^{p-1}/p & \mbox{if}\ 2\nmid f,\\
			(-1)^{(p-1)/f}G_p(\psi)^{p-1}/p & \mbox{if}\ 2\mid f\ \text{and}\ \psi(-1)=1,\\
			(-1)^{(f+2)(p-1)/(2f)}G_p(\psi)^{p-1}/p & \mbox{if}\ 2\mid f\ \text{and}\ \psi(-1)=-1.
		\end{cases}
	\end{equation}
	
	Along this line, Chapman further studied some variants of $C_p(\psi)$. For instance, Chapman \cite{Chapman} considered the matrix
	\begin{equation*}
		V_p=\left[\left(\frac{i+j-1}{p}\right)\right]_{1\le i,j\le (p-1)/2}.
	\end{equation*}
	Observing that 
	\begin{equation*}
		\left(\frac{p+1}{2}-i\right)+\left(\frac{p+1}{2}-j\right)\equiv -(i+j-1)\pmod p,
	\end{equation*}
	one can verify that 
	\begin{equation*}
		\det V_p=\det \left[\left(\frac{i+j-1}{p}\right)\right]_{1\le i,j\le (p-1)/2}
		=\left(\frac{-1}{p}\right)\det \left[\left(\frac{i+j}{p}\right)\right]_{1\le i,j\le (p-1)/2}.
	\end{equation*}

	Although the structures of $C_p(\psi)$ and $V_p$ are very similar, the calculation of $\det V_p$ is much more complicated than that of $\det C_p(\psi)$. Surprisingly, $\det V_p$ is closely related to the real quadratic field $\mathbb{Q}(\sqrt{p})$.	In fact, let $\varepsilon_p>1$ and $h_p$ be the fundamental unit and the class number of $\mathbb{Q}(\sqrt{p})$ respectively and write 
	\begin{equation*}
		\varepsilon_p^{h_p}=a_p+b_p\sqrt{p}\ (a_b,b_p\in\mathbb{Q}).
	\end{equation*}
	 Chapman \cite{Chapman} proved that 
	\begin{equation}
		\det V_p=
		\begin{cases}
			(-1)^{(p-1)/4}2^{(p-1)/2}b_p & \mbox{if}\ p\equiv 1\pmod 4,\\
			0                            & \mbox{if}\ p\equiv 3\pmod 4.
		\end{cases}
	\end{equation}
	
	After introducing the above relevant research results, we now describe our research motivations. 
	
	For any complex number $z$ with Re$(z)>0$, the Gamma function is defined by 
	\begin{equation}\label{Eq. definition of the Gamma function}
		\Gamma(z)=\int_{0}^{+\infty}t^{z-1}e^{-t}dt.
	\end{equation}
	The determinants involving Gamma function have been extensively studied and have many applications in probability theory and mathematical physics. For example, given a positive integer $n$, by \cite[(4.5) and (4.8)]{N} we have 
	\begin{equation}\label{Eq. example 1 for det of Gamma}
		\det\left[\Gamma(i+j)\right]_{1\le i,j\le n}=\prod_{r=0}^{n-1}r!(r+1)!,
	\end{equation}
	and 
	\begin{equation}\label{Eq. example 2 for det of Gamma}
		\det\left[\frac{1}{\Gamma(i+j)}\right]_{1\le i,j\le n}=(-1)^{\frac{n(n-1)}{2}}\prod_{r=0}^{n-1}\frac{r!}{(n+r)!}.
	\end{equation}

	Also, the well known Gauss multiplication formula (see \cite[Theorem 1.5.2]{A}) says that for any $m\in\mathbb{Z}^+$ and any complex number $z$ we have 
	\begin{equation}\label{Eq. the Gauss multiplication formula}
		(2\pi)^{(m-1)/2}\Gamma(mz)=m^{(mz-1/2)}\prod_{j=0}^{m-1}\Gamma\left(z+\frac{j}{m}\right).
	\end{equation}
	The finite field counterpart of (\ref{Eq. the Gauss multiplication formula}) is the Hasse-Davenport product formula (see \cite[Theorem 11.3.5]{B}), which states that for any $\rho\in\widehat{\mathbb{F}_q^{\times}}$ with $\ord(\rho)=m$ and any $\psi\in\widehat{\mathbb{F}_q^{\times}}$, we have 
	\begin{equation}\label{Eq. the HD product formula}
		\prod_{0\le a\le m-1}G_q(\psi\rho^a)=-\psi^{-m}(m)G_q(\psi^m)\prod_{0\le a\le m-1}G_q(\rho^a).
	\end{equation}
	
	In view of (\ref{Eq. definition of the Gauss sum}), (\ref{Eq. definition of the Gamma function}), (\ref{Eq. the Gauss multiplication formula}) and (\ref{Eq. the HD product formula}), the function 
	\begin{equation*}
		G_q(\cdot):\ \widehat{\mathbb{F}_q^{\times}}\rightarrow \mathbb{Q}(\zeta_{q-1},\zeta_p)
	\end{equation*}
by sending the character $\psi$ to the Gauss sum $G_q(\psi)$, is indeed a finite field analogue of the Gamma function $\Gamma(\cdot)$. Readers may refer to \cite{F} for the detailed introduction on this topic.

Throughout the remaining part of this paper, we let $\chi$ be a generator of $\widehat{\mathbb{F}_q^{\times}}$. Now motivated by (\ref{Eq. example 1 for det of Gamma}), (\ref{Eq. example 2 for det of Gamma}) and the structure of Carlitz's cyclotomic matrix $C_p(\psi)$ defined by (\ref{Eq. definition of Carlitz's matrix C(psi)}) , it is natural to consider the matrix 
	\begin{equation}\label{Eq. definition of Aq(1)}
		A_q(1):=\left[G_q(\chi^{i+j})\right]_{0\le i,j\le q-2}.
	\end{equation}

	On the other hand, Sun \cite{Sun} studied the matrix 
	\begin{equation*}\label{Eq. Sun's matrix Sp}
		S_p=\left[\left(\frac{i^2+j^2}{p}\right)\right]_{1\le i,j\le (p-1)/2},
	\end{equation*}
	which involves the nonzero squares of $\mathbb{F}_p$. If we treat $S_p$ as a matrix over $\mathbb{F}_p$, then Sun \cite[Theorem 1.2]{Sun} proved that $-\det S_p$ is always a nonzero square of $\mathbb{F}_p$. Moreover, as a matrix over $\mathbb{Z}$, Sun conjectured that if $p\equiv 3\pmod 4$, then $-\det S_p$ is indeed a square of some integer. This conjecture was later confirmed by Alekseyev and Krachun. For the case $p\equiv 1\pmod 4$, it is known that there exist $b\in\mathbb{Z}$ and a unique $a\in\mathbb{Z}$ such that $p=a^2+4b^2$ and $a\equiv 1\pmod 4$. Cohen, Sun and Vsemirnov \cite[Remark 4.2]{Sun} conjectured that $\det S_p/a$ is also a square of some integer. This conjecture was later confirmed by the first author \cite{Wu}. For the recent progress on this topic, readers may refer to \cite{Grinberg, krachun, W21}.
	
	In the isomorphism $\mathbb{F}_q^{\times}\cong\widehat{\mathbb{F}_q^{\times}}$, the nonzero squares of $\mathbb{F}_q$ correspond to the even characters $\chi^{2i}$ $(i=0,1,\cdots,(q-3)/2)$. Thus, inspired by Sun's matrix $S_p$ and the construction of $A_q(1)$, it is also natural to consider the matrix 
	\begin{equation}\label{Eq. definition of Aq(2)}
		A_q(2):=\left[G_q(\chi^{2i+2j})\right]_{0\le i,j\le (q-3)/2}.
	\end{equation} 
	
	Motivated by the above results, in general, for any positive integer $k\mid q-1$, we define 
	\begin{equation}\label{Eq. definition of general }
		A_q(k):=\left[G_q(\chi^{ki+kj})\right]_{0\le i,j\le (q-1-k)/k},
	\end{equation}
    and 
    \begin{equation}
    	B_q(k):=\left[G_q(\chi^{ki+kj})^{-1}\right]_{0\le i,j\le (q-1-k)/k},
    \end{equation}
	which can be naturally viewed as the finite field analogues of (\ref{Eq. example 1 for det of Gamma}) and (\ref{Eq. example 2 for det of Gamma}) respectively. 
	
	\subsection{Main Results} We now state our main results of this paper.
	
	\begin{theorem}\label{Thm. general Aq(k)}
		Let $q=p^n$ be a prime power with $p$ prime and $n\in\mathbb{Z}^+$. Let $\chi$ be a generator of $\widehat{\mathbb{F}_q^{\times}}$. Let $k\mid q-1$ be a positive integer and let $m=(q-1)/k$. Then the following hold. 
		
		{\rm (i)} $\det A_q(k)\in\mathbb{Z}$, $\det B_q(k)\in\mathbb{Q}$ and both are independent of the choice of the generator $\chi$.
		
		{\rm (ii)} We have the congruence 
		\begin{equation*}
			\det A_q(k)\equiv (-1)^{\frac{m^2-m+2}{2}}\pmod p.
		\end{equation*} 
	\end{theorem}
	
	The next theorem gives the explicit values of $\det A_q(1)$ and $\det A_q(2)$. 
	
	\begin{theorem}\label{Thm. Aq(1) and Aq(2)}
		Let $q=p^n$ be a prime power with $p$ prime and $n\in\mathbb{Z}^+$. Let $\chi$ be a generator of $\widehat{\mathbb{F}_q^{\times}}$. Then the following results hold.
		
		{\rm (i)} For the cyclotomic matrix $A_q(1)$, we have 
		\begin{equation*}
			\det A_q(1)=(-1)^{\frac{(q-2)(q-3)}{2}}(q-1)^{q-1}.
		\end{equation*}
	
		{\rm (ii)} If $q\equiv 1\pmod2$, then 
		\begin{equation*}
			\det A_q(2)=(-1)^{\alpha_n}\left(\frac{q-1}{2}\right)^{\frac{q-1}{2}}2^{\frac{p^{n-1}-1}{2}},
		\end{equation*}
	where 
	    \begin{equation*}
	    	\alpha_n=
	    	\begin{cases}
	    		1            & \mbox{if}\ n\equiv 1\pmod 2,\\
	    		(p^2+7)/8    & \mbox{if}\ n\equiv 0\pmod 2.
	    	\end{cases}
	    \end{equation*}
	\end{theorem}
	
	\begin{remark}
		In fact, for any positive integer $k\mid q-1$, we can prove that (see Lemma \ref{Lem. eigenvalues of Cq(k)})
		$$\det A_q(k)=(-1)^{\frac{(m-1)(m-2)}{2}}\prod_{b\in\mathbb{F}_q^{\times}/U_k}\lambda_b,$$
		where $m=(q-1)/k$, $U_k=\{x\in\mathbb{F}_q:\ x^k=1\}$ and 
		$$\lambda_b=m\sum_{y\in U_k}\zeta_p^{\Tr(by)}.$$
		However, for $3\le k<q-1$, finding a simple expression of $\det A_q(k)$ like the case $k\in\{1,2\}$ seems very difficult.
	\end{remark}
	
	For any positive integer $k\mid q-1$, let 
	$$o_k(p):=\min\{f\in\mathbb{Z}^+:\ p^f\equiv 1\pmod k\}$$
	be the order of $p$ modulo $k$. As $p^n\equiv 1\pmod k$, we clearly have $o_k(p)\mid n$. 
	
	Now we state our next result, which concerns $\det B_q(k)$. 
	
	\begin{theorem}\label{Thm. Bq(1) and Bq(2)}
	Let $q=p^n$ be a prime power with $p$ prime and $n\in\mathbb{Z}^+$. Let $\chi$ be a generator of $\widehat{\mathbb{F}_q^{\times}}$ and let $k\mid q-1$ be a positive integer. Then the following results hold.
	
	{\rm (i)} For the singularity of $B_q(k)$ we have 
	$$B_q(k)\ \text{is a nonsingular matrix}\ \Leftrightarrow o_k(p)=n.$$
	In particular, $\det B_p(k)\neq 0$ for any $k\mid p-1$. Also, $\det B_{q}(1)=0$ whenever $n\ge2$, and $\det B_q(2)=0$ if $2\nmid q$ and $q$ is not a prime. 
	 
	{\rm (ii)} For the case $k=1$ and $n=1$ we have 
		\begin{equation*}
			\det B_p(1)=\frac{(-1)^{\frac{p(p+1)}{2}}(p-1)^{p-1}}{p^{p-2}}.
		\end{equation*}
	
	{\rm (iii)} If $p$ is an odd prime, then 
	\begin{equation*}
		\det B_p(2)=(-1)^{\frac{(p+3)(p-1)}{4}}p\left(\frac{p-1}{2p}\right)^{\frac{p-1}{2}}.
	\end{equation*}
	\end{theorem}

	\subsection{Outline of This Paper} We will prove Theorem \ref{Thm. general Aq(k)} in Section 2. The proofs of Theorems \ref{Thm. Aq(1) and Aq(2)}--\ref{Thm. Bq(1) and Bq(2)} will be given in Sections 3--4 respectively.

	\section{Proof of Theorem \ref{Thm. general Aq(k)}}
	
		Let $q=p^n$ be a prime power with $p$ prime and $n\in\mathbb{Z}^+$. Let $L=\mathbb{Q}(\zeta_{q-1},\zeta_p)$ and let $\O_L$ be the ring of algebraic integers over $L$. Let $\p$ be a prime ideal of $\O_L$ such that $p\in\p$. By algebraic number theory, it is easy to verify that 
		$$\O_L/\p\cong\mathbb{F}_q.$$
	    Now let $\chi_{\p}\in\widehat{\mathbb{F}_q^{\times}}$ be the Teich\"{u}muller character of $\p$, i.e., 
	    $$\chi_{\p}(x\mod\p)\equiv x\pmod {\p}$$
	    for any $x\in\O_L$. One can verify that $\chi_{\p}$ is a generator of $\widehat{\mathbb{F}_q^{\times}}$. Also, for any integer $0\le r\le q-2$, let 
	    \begin{equation*}
	    	r=\sum_{0\le j\le n-1}r_jp^j\ (0\le r_j\le p-1)
	    \end{equation*}
	    be the decomposition of $r$ in base $p$. We define 
	    $$s(r):=\sum_{0\le j\le n-1}r_j,$$
	    and 
	    $$t(r):=\prod_{0\le j\le n-1}(r_j!).$$
	    
	    We begin with the well known Stickelberger congruence (see \cite[Theorem 3.6.6]{Cohen}). 
	    
	    \begin{lemma}\label{Lem. the Stickelberger congruence}
	    	Let notations be as above. Then for any integer $0\le r\le q-2$ we have 
	    	\begin{equation*}
	    	  \frac{G_q(\chi_{\p}^{-r})}{(\zeta_p-1)^{s(r)}}\equiv -\frac{1}{t(r)}\pmod{\p}.
	    	\end{equation*}
    	In particular, $G_q(\chi_{\p}^{-r})\equiv 0 \pmod {\p}$ whenever $1\le r\le q-2$. 
	    \end{lemma}
	
	 Let $m$ be a positive integer. We next consider a permutation of $\mathbb{Z}/m\mathbb{Z}$. Let $a\in\mathbb{Z}$ with $\gcd(a,m)=1$. Then the map $x\mod m\mapsto ax\mod m$ induces a permutation $\tau_m(a)$ of $\mathbb{Z}/m\mathbb{Z}$. Lerch \cite{Lerch} determined the sign of this permutation.
	 
	 \begin{lemma}\label{Lem. the Lerch permutation}
	 	Let notations be as above and let $\sign{\tau_m(a)}$ denote the sign of $\tau_m(a)$. Then 
	 	\begin{equation*}
	 		\sign{\tau_m(a)}=
	 		\begin{cases}
	 			(\frac{a}{m})  & \mbox{if}\ m\equiv 1\pmod 2,\\
	 			1              & \mbox{if}\ m\equiv 2\pmod 4,\\
	 			(-1)^{(a-1)/2} & \mbox{if}\ m\equiv 0\pmod 4,
	 		\end{cases}
	 	\end{equation*}
 	where $(\frac{\cdot}{m})$ is the Jacobi symbol if $m$ is odd. In particular, we have 
 	\begin{equation}\label{Eq. the sign of -1}
 		\sign{\tau_{m}(-1)}=(-1)^{\frac{(m-1)(m-2)}{2}}.
 	\end{equation}
	 \end{lemma}
	
	For simplicity, $\sum_{a\in\mathbb{F}_q}$ and $\sum_{a\in\mathbb{F}_q^{\times}}$ will be abbreviated as $\sum_{a}$ and $\sum_{a\neq 0}$ respectively. Also, let $M_m(\O_L)$ denote the set of $m\times m$ matrices with all entries contained in $\O_L$. For any $M_1,M_2\in M_m(\O_L)$, the symbol 
	$$M_1\equiv M_2\pmod{\p}$$
	denotes that every entry of $M_1-M_2$ is contained in the prime ideal $\p$. 
	
	Now we are in a position to prove our first theorem. 
	
	{\noindent{\bf Proof of Theorem \ref{Thm. general Aq(k)}.}} (i) Recall that $L=\mathbb{Q}(\zeta_{q-1},\zeta_p)$. As $\gcd(q-1,p)=1$, we have 
	$\mathbb{Q}(\zeta_{q-1})\cap\mathbb{Q}(\zeta_p)=\mathbb{Q}$ and hence 
	\begin{equation*}
		\Gal(L/\mathbb{Q})\cong\Gal(L/\mathbb{Q}(\zeta_{q-1}))\oplus\Gal(L/\mathbb{Q}(\zeta_{p}))\cong \left(\mathbb{Z}/p\mathbb{Z}\right)^{\times}\oplus\left(\mathbb{Z}/(q-1)\mathbb{Z}\right)^{\times}.
	\end{equation*}
    For any $\sigma_{l,s}\in\Gal(L/\mathbb{Q})$ with $l,s\in\mathbb{Z}$, $\gcd(l,p)=1$ and $\gcd(s,q-1)=1$, we have $\sigma_{l,s}(\zeta_{p})=\zeta_p^l$ and $\sigma_{l,s}(\zeta_{q-1})=\zeta_{q-1}^s$. 
    
    Now for any $\psi\in\widehat{\mathbb{F}_q^{\times}}$, one can verify that 
    \begin{align*}
    	\sigma_{l,s}\left(G_q(\psi)\right)
    	&=\sigma_{l,s}\left(\sum_{a}\psi(a)\zeta_p^{\Tr(a)}\right)\\
    	&=\sum_{a}\psi^s(a)\zeta_{p}^{\Tr(la)}\\
    	&=\psi^{-s}(l)\sum_{a}\psi^s(la)\zeta_{p}^{\Tr(la)}\\
    	&=\psi^{-s}(l)\sum_{a}\psi^s(a)\zeta_{p}^{\Tr(a)}\\
    	&=\psi^{-s}(l)G_q(\psi^s).
    \end{align*}
	Letting $m=(q-1)/k$, by the above and Lemma \ref{Lem. the Lerch permutation} we have 
	\begin{align*}
		\sigma_{l,s}\left(\det A_q(k)\right)
		&=\det\left[\chi^{-(ksi+ksj)}(l)G_q(\chi^{ksi+ksj})\right]_{0\le i,j\le m-1}\\
		&=\left(\prod_{0\le r\le m-1}\chi^{-ksr}(l)\right)^2\cdot \sign{\tau_{m}(s)}^2\cdot \det A_q(k)\\
		&=\det A_q(k).
	\end{align*}
	Noting that $\gcd(s,m)=1$, the last equality follows from 
	\begin{align*}
		\left(\prod_{0\le r\le m-1}\chi^{-ksr}(l)\right)^2
		&=\prod_{0\le r\le m-1}\chi^{kr}(l)\cdot \prod_{0\le r\le m-1}\chi^{-kr}(l)=1.
	\end{align*}
By the Galois theory, we have $\det A_q(k)\in\mathbb{Q}$. Note that $\det A_q(k)$ is an algebraic integer. Hence $\det A_q(k)\in\mathbb{Z}$. 

Also, for any integer $s$, the character $\chi^s$ is a generator of $\widehat{\mathbb{F}_q^{\times}}$ if and only if $\gcd(s,q-1)=1$. Hence for any generator $\chi'$, there is an integer $s$ with $\gcd(s,q-1)=1$ such that 
\begin{equation*}
	\det\left[G_q(\chi'^{(ki+kj)})\right]_{0\le i,j\le m-1}=\sigma_{1,s}\left(\det A_q(k)\right)=\det A_q(k).
\end{equation*}

In view of the above, we have proved (i) for $\det A_q(k)$. Using essentially the same method, one can also verify that (i) holds for $\det B_q(k)$. 
	
	(ii) By (\ref{Eq. the sign of -1}) of Lemma \ref{Lem. the Lerch permutation} we have 
	\begin{equation*}
		\det A_q(k)=(-1)^{\frac{(m-1)(m-2)}{2}}\det \left[G_q(\chi^{ki-kj})\right]_{0\le i,j\le m-1}.
	\end{equation*}
	By Lemma \ref{Lem. the Stickelberger congruence}, for any $0\le i\neq j\le m-1$ we have 
	$$G_q(\chi^{ki-kj})\equiv 0\pmod{\p}.$$
	By this and noting that $G_q(\varepsilon)=-1$, we obtain 
	\begin{equation*}
		\left[G_q(\chi^{ki-kj})\right]_{0\le i,j\le m-1}\equiv -I_m\pmod{\p},
	\end{equation*}
	where $I_m$ is the $m\times m$ identity matrix. By the above, we have 
	\begin{equation*}
		\det A_q(k)\equiv (-1)^{\frac{(m-1)(m-2)}{2}+m}\equiv (-1)^{\frac{m^2-m+2}{2}}\pmod p.
	\end{equation*}
	This completes the proof of (ii). 
	
	 In view of the above, we have completed the proof of Theorem \ref{Thm. general Aq(k)}.\qed

	\section{Proof of Theorem \ref{Thm. Aq(1) and Aq(2)}}

	We begin with follow known result in linear algebra. 
	
	\begin{lemma}\label{Lem. an easy lemma in linear algebra}
		Let $m$ be a positive integer and let $M$ be an $m\times m$ complex matrix. Let $\lambda_1,\cdots,\lambda_m\in\mathbb{C}$, and let $\v_1,\cdots,\v_m\in\mathbb{C}^m$ be column vectors. Suppose that 
		$$M\v_i=\lambda_i\v_i$$
		for each $1\le i\le m$ and that the vectors $\v_1,\cdots,\v_m$ are linearly independent over $\mathbb{C}$. Then $\lambda_1,\cdots,\lambda_m$ are exactly all the eigenvalues of $M$ (counting multiplicity). 
	\end{lemma}
	
	Recall that $\chi$ is a generator of $\widehat{\mathbb{F}_q^{\times}}$ and $k\mid q-1$ is a positive integer. Throughout the remaining part of this section, we let $m=(q-1)/k$. Let $U_k:=\{x\in\mathbb{F}_q:\ x^k=1\}$ be the set of all $k$th roots of unity in $\mathbb{F}_q$. 
	
	We also need the following known result. 
	
	\begin{lemma}\label{Lem. sums involving kth characters}
		Let notation be as above. Then for any $x\in\mathbb{F}_q$ we have 
		\begin{equation*}
			\sum_{0\le r\le m-1}\chi^{kr}(x)=
			\begin{cases}
				m  & \mbox{if}\ x\in U_k,\\
				0  & \mbox{otherwise.}
			\end{cases}
		\end{equation*}
	\end{lemma}

	Define 
	\begin{equation}\label{Eq. definition of Cq(k)}
		C_q(k):=\left[G_q(\chi^{ki-kj})\right]_{0\le i,j\le m-1}.
	\end{equation}
The next lemma determines all the eigenvalues of $C_q(k)$. 
	
	\begin{lemma}\label{Lem. eigenvalues of Cq(k)}
		Let notations be as above. For any element $b$ in the quotient group $\mathbb{F}_q^{\times}/U_k$, let 
		\begin{equation}\label{Eq. definition of lambda b}
			\lambda_b:=m\sum_{y\in U_k}\zeta_p^{\Tr(by)}.
		\end{equation}
	Then these $\lambda_b$ (where $b\in\mathbb{F}_q^{\times}/U_k$) are exactly all the eigenvalues of $C_q(k)$. 
	\end{lemma}
	
	\begin{proof}
		For any $b\in\mathbb{F}_q^{\times}/U_k$, define the column vector 
		$$\v_b:=\left(\chi^{0}(b),\chi^{k}(b),\cdots,\chi^{k(m-1)}(b)\right)^T.$$
		Then by Lemma \ref{Lem. sums involving kth characters}  
		\begin{align*}
			  \sum_{0\le r\le m-1}G_q(\chi^{ki-kr})\chi^{kr}(b)
			&=\sum_{0\le r\le m-1}\sum_{a\neq 0}\chi^{ki}(a)\zeta_{p}^{\Tr(a)}\chi^{kr}\left(\frac{b}{a}\right)\\
			&=\sum_{a\neq 0}\chi^{ki}(a)\zeta_{p}^{\Tr(a)}\sum_{0\le r\le m-1}\chi^{kr}\left(\frac{b}{a}\right)\\
			&=m\sum_{y\in U_k}\chi^{ki}(by)\zeta_{p}^{\Tr(by)}\\
			&=\lambda_b\chi^{ki}(b).
		\end{align*}
	This implies $C_q(k)\v_b=\lambda_b\v_b$ for each $b\in\mathbb{F}_q^{\times}/U_k$. 
	
	On the other hand, let 
	$$V_k=\left[\chi^{ki}(b)\right]_{0\le i\le m-1,\ b\in\mathbb{F}_q^{\times}/U_k}.$$
	Then it is clear that 
	\begin{equation*}
		(\det V_k)^2=\pm\prod_{b,b'\in\mathbb{F}_q^{\times}/U_k, b\neq b'}\left(\chi^k(b)-\chi^k(b')\right)\neq 0.
	\end{equation*}
    Hence these $\v_b$ are linearly independent over $\mathbb{C}$. Now by Lemma \ref{Lem. an easy lemma in linear algebra}, these $\lambda_b$ (where $b\in\mathbb{F}_q^{\times}/U_k$) are exactly all the eigenvalues of $C_q(k)$. This completes the proof. 
	\end{proof}
	
	We also need the  following result concerning the trace map $\Tr:=\Tr_{\mathbb{F}_q/\mathbb{F}_p}$ (see \cite[Proposition 2.4.11]{Cohen}). 
	
	\begin{lemma}\label{Lem. the trace map}
		The trace map $\Tr$ is a surjective $\mathbb{F}_p$-homomorphism from $\mathbb{F}_q$ to $\mathbb{F}_p$. Also, 
		$$\Ker(\Tr)=\{x\in\mathbb{F}_q:\ \Tr(x)=0\}=\{x^p-x:\ x\in\mathbb{F}_q\}$$
		is an $\mathbb{F}_p$-subspace of dimension $n-1$. 
	\end{lemma}
	
	Now we are in a position to prove our second theorem.
	
	{\noindent{\bf Proof of Theorem \ref{Thm. Aq(1) and Aq(2)}.}} (i) By Lemma \ref{Lem. the Lerch permutation} and Lemma \ref{Lem. eigenvalues of Cq(k)}, one can verify that $\det A_q(1)$ is equal to 
	\begin{align*}
		 \sign{\tau_{q-1}(-1)}\cdot\det C_q(1)
		&=(-1)^{\frac{(q-2)(q-3)}{2}}(q-1)^{q-1}\prod_{b\in\mathbb{F}_q^{\times}}\zeta_{p}^{\Tr(b)}\\
		&=(-1)^{\frac{(q-2)(q-3)}{2}}(q-1)^{q-1}.
	\end{align*}
	The last equality follows from 	$\sum_{b\in\mathbb{F}_q^{\times}}\Tr(b)=0$. 
	
	(ii) Note that $2\nmid p$. We first consider $\det C_q(2)$ defined by (\ref{Eq. definition of Cq(k)}). By Lemma \ref{Lem. eigenvalues of Cq(k)} we have 
	\begin{equation}\label{Eq. 1 in the proof of Thm. 2}
		\det C_q(2)=\left(\frac{q-1}{2}\right)^{\frac{q-1}{2}}\prod_{b\in\mathbb{F}_q^{\times}/U_2}\left(\zeta_p^{\Tr(b)}+\zeta_p^{\Tr(-b)}\right).
	\end{equation}
	Let 
	$$\xi=\prod_{b\in\mathbb{F}_q^{\times}/U_2}\left(\zeta_p^{\Tr(b)}+\zeta_p^{\Tr(-b)}\right).$$
	Then by Lemma \ref{Lem. the trace map} we have 
	\begin{align*}
		 \xi^2
		&=\prod_{b\in\mathbb{F}_q^{\times}/U_2}\left(\zeta_p^{\Tr(b)}+\zeta_p^{\Tr(-b)}\right)^2\\
		&=\prod_{b\in\mathbb{F}_q^{\times}}\left(\zeta_p^{\Tr(b)}+\zeta_p^{\Tr(-b)}\right)\\
		&=\prod_{b\in\mathbb{F}_q^{\times}}\zeta_{p}^{\Tr(-b)}\prod_{b\in\mathbb{F}_q^{\times}}\left(1+\zeta_p^{\Tr(2b)}\right)\\
		&=\prod_{b\in\mathbb{F}_q^{\times}}\left(1+\zeta_p^{\Tr(b)}\right)\\
		&=\prod_{b\in\Ker(\Tr)\setminus\{0\}}2\times\prod_{b\in\mathbb{F}_q^{\times}\setminus\Ker(\Tr)}\left(1+\zeta_p^{\Tr(b)}\right)\\
		&=2^{p^{n-1}-1}\prod_{b\in\mathbb{F}_q^{\times}\setminus\Ker(\Tr)}\left(1+\zeta_p^{\Tr(b)}\right).
	\end{align*}
	Note that 
	\begin{align*}
			\prod_{b\in\mathbb{F}_q^{\times}\setminus\Ker(\Tr)}\left(1+\zeta_p^{\Tr(b)}\right)
			&=\prod_{b\in\mathbb{F}_q^{\times}\setminus\Ker(\Tr)}\frac{1-\zeta_p^{\Tr(2b)}}{1-\zeta_p^{\Tr(b)}}\\
			&=\prod_{b\in\mathbb{F}_q^{\times}\setminus\Ker(\Tr)}\frac{1-\zeta_p^{\Tr(b)}}{1-\zeta_p^{\Tr(b)}}\\
			&=1.
	\end{align*}
     In view of the above, we obtain 
	$$\xi=\pm 2^{\frac{p^{n-1}-1}{2}}$$
	and hence by (\ref{Eq. 1 in the proof of Thm. 2}) we have 
	\begin{equation*}
		\det C_q(2)=\pm \left(\frac{q-1}{2}\right)^{\frac{q-1}{2}}\cdot 2^{\frac{p^{n-1}-1}{2}}.
	\end{equation*}
	By Lemma \ref{Lem. the Lerch permutation} we also have 
	\begin{equation*}
		\det A_q(2)=\pm \left(\frac{q-1}{2}\right)^{\frac{q-1}{2}}\cdot 2^{\frac{p^{n-1}-1}{2}}.
	\end{equation*}
	Recall that $q=p^n$. Since $p>2$, it is easy to verify that 
	\begin{equation*}
		\left(\frac{q-1}{2}\right)^{\frac{q-1}{2}}\cdot 2^{\frac{p^{n-1}-1}{2}}\equiv (-1)^{\frac{q-1}{2}}\left(\frac{2}{p}\right)\equiv (-1)^{\frac{q-1}{2}+\frac{p^2-1}{8}}\pmod p.
	\end{equation*}
	By Theorem \ref{Thm. general Aq(k)}(ii) we have 
	$$\det A_q(2)\equiv (-1)^{(q-1)(q-3)/8 +1}\pmod p. $$
	Combining the above results and by some routine calculations, one can verify that 
		\begin{equation*}
		\det A_q(2)=(-1)^{\alpha_n}\left(\frac{q-1}{2}\right)^{\frac{q-1}{2}}2^{\frac{p^{n-1}-1}{2}},
	\end{equation*}
	where 
	\begin{equation*}
		\alpha_n=
		\begin{cases}
			1            & \mbox{if}\ n\equiv 1\pmod 2,\\
			(p^2+7)/8    & \mbox{if}\ n\equiv 0\pmod 2.
		\end{cases}
	\end{equation*}
  This completes the proof of Theorem \ref{Thm. Aq(1) and Aq(2)}. \qed 
	
	\section{Proof of Theorem \ref{Thm. Bq(1) and Bq(2)}}
	Recall that $q=p^n$ and $k\mid q-1$ is a positive integer. In this section, we also let $m=(q-1)/k$. We begin with the following known result (see \cite[Theorem 1.1.4]{B}).
	\begin{lemma}\label{Lem. relation between Gauss sums}
		Let $\psi\in\widehat{\mathbb{F}_q^{\times}}$ be a nontrivial character. Then 
		\begin{equation*}
			G_q(\psi)G_q(\psi^{-1})=\psi(-1)q.
		\end{equation*}
	\end{lemma}
	
	Using Lemma \ref{Lem. relation between Gauss sums}, it is easy to verify the following result.
	
	\begin{lemma}\label{Lem. transformation of Bq(k)}
		Let notations be as above. Then
		\begin{equation*}
			\left[\frac{1}{G_q(\chi^{-ki+kj})}\right]_{0\le i,j\le m-1}=\frac{1-q}{q}I_m+\frac{1}{q}D_q(k),
		\end{equation*} 
	where
	\begin{equation}\label{Eq. definition of Dq(k)}
		D_q(k):=\left[(-1)^{ki-kj}G_q(\chi^{ki-kj})\right]_{0\le i,j\le m-1}.
	\end{equation}
 Moreover, we have 
 \begin{equation*}
 	P_k^{-1}C_q(k)P_k=D_q(k),
 \end{equation*}
where $P_k$ is the diagonal matrix 
$$\diag\left((-1)^0,(-1)^{k},\cdots,(-1)^{(m-1)k}\right),$$ 
and $C_q(k)$ is defined by (\ref{Eq. definition of Cq(k)}). Hence $D_q(k)$ and $C_q(k)$ have the same eigenvalues.
	\end{lemma} 
	
	{\noindent{\bf Proof of Theorem \ref{Thm. Bq(1) and Bq(2)}.}} (i) Since $\cos\frac{2\pi\Tr(by)}{p}\le 1$, we clearly have 
	\begin{equation}\label{Eq. sum=0 in the proof of Thm. 3}
		\sum_{y\in U_k}\left(\zeta_p^{\Tr(by)}-1\right)=0\Leftrightarrow \Tr(by)=0\ \text{for any}\ y\in U_k.
	\end{equation}
	By Lemma \ref{Lem. the Lerch permutation} 
	\begin{equation*}
		\det B_q(k)=\sign{\tau_m(-1)}\cdot\det\left[G_q(\chi^{-ki+kj})^{-1}\right]_{0\le i,j\le m-1}. 
	\end{equation*}
	Hence by Lemma \ref{Lem. transformation of Bq(k)}, Lemma \ref{Lem. eigenvalues of Cq(k)} and (\ref{Eq. sum=0 in the proof of Thm. 3}), it is clear that 
	\begin{align}
		\det B_q(k)\neq 0
		&\Leftrightarrow \frac{1-q}{q}+\frac{1}{q}\lambda_b\neq 0\ (\forall b\in\mathbb{F}_q^{\times}/U_k) \notag \\ &\Leftrightarrow\sum_{y\in U_k}\left(\zeta_p^{\Tr(by)}-1\right)\neq 0\ (\forall b\in\mathbb{F}_q^{\times}) \notag \\
		&\Leftrightarrow \{b\in\mathbb{F}_q^{\times}:\ \Tr(by)=0\ \text{for any}\ y\in U_k\}=\emptyset.\label{Eq. the condition for Bq(k) being nonsingular}
	\end{align}
	
	"$\Leftarrow$". Suppose $o_k(p)=n$. Then the smallest extension of $\mathbb{F}_p$ containing $U_k$ is $\mathbb{F}_q$, i.e., 
	\begin{equation}\label{Eq. the smallest field containing Uk}
			\mathbb{F}_p\left(U_k\right)=\mathbb{F}_q.
	\end{equation}
	By the linearity of the trace map and (\ref{Eq. the smallest field containing Uk}), we have 
	\begin{equation*}
		\{b\in\mathbb{F}_q^{\times}:\ \Tr(by)=0\ \text{for any}\ y\in U_k\}=\{b\in\mathbb{F}_q^{\times}:\ \Tr(by)=0\ \text{for any}\ y\in \mathbb{F}_q\}=\emptyset.
	\end{equation*}
	Hence by (\ref{Eq. the condition for Bq(k) being nonsingular}) the matrix $B_q(k)$ is nonsingular.
	
	"$\Rightarrow$". Suppose $\det B_q(k)\neq 0$. If $o_k(p)=r<n$, then the smallest extension of $\mathbb{F}_p$ containing $U_k$ is the field $\mathbb{F}_{p^r}\subsetneq\mathbb{F}_q$. By \cite[Proposition 2.4.11]{Cohen}, the trace map $\Tr_{\mathbb{F}_q/\mathbb{F}_{p^r}}$ is surjective and $\Ker(\Tr_{\mathbb{F}_q/\mathbb{F}_{p^r}})$ is an $\mathbb{F}_{p^r}$-subspace of dimension $n/r-1$. As $r<n$, there is an element $b\in\mathbb{F}_q^{\times}$ such that $\Tr_{\mathbb{F}_q/\mathbb{F}_{p^r}}(b)=0$. Hence for any $y\in U_k\subseteq\mathbb{F}_{p^r}$, we have 
	\begin{align*}
		\Tr(by)
		&=\Tr_{\mathbb{F}_{p^r}/\mathbb{F}_p}\circ\Tr_{\mathbb{F}_q/\mathbb{F}_{p^r}}(by)\\
		&=\Tr_{\mathbb{F}_{p^r}/\mathbb{F}_p}\left(y\Tr_{\mathbb{F}_q/\mathbb{F}_{p^r}}(b)\right)\\
		&=0.
	\end{align*}
	This, together with (\ref{Eq. the condition for Bq(k) being nonsingular}), implies that $\det B_q(k)=0$, which is a contradiction. Hence we must have $o_k(p)=n$. This completes the proof of (i). 
	
	(ii) Observe first that 
	\begin{equation}\label{Eq. an easy evaluation for roots of unity}
		\prod_{b\in\mathbb{F}_p^{\times}}\left(1-\zeta_p^b\right)=\lim_{x\rightarrow 1}\frac{x^p-1}{x-1}=p.
	\end{equation}
	Using this and by Lemma \ref{Lem. transformation of Bq(k)}, Lemma \ref{Lem. eigenvalues of Cq(k)} and (\ref{Eq. definition of lambda b}), we have 
	\begin{align*}
		\det [G_p(\chi^{-i+j})^{-1}]_{0\le i,j\le p-2}
		&=\prod_{b\in\mathbb{F}_p^{\times}}\left(\frac{1-p}{p}+\frac{p-1}{p}\zeta_p^b\right)\\
		&=\left(\frac{1-p}{p}\right)^{p-1}\prod_{b\in\mathbb{F}_p^{\times}}(1-\zeta_p^b)\\
		&=\left(\frac{1-p}{p}\right)^{p-1}p.
	\end{align*}
	By Lemma \ref{Lem. the Lerch permutation} we have 
	$$\det B_p(1)=\frac{(-1)^{\frac{p(p+1)}{2}}(p-1)^{p-1}}{p^{p-2}}. $$
	
	(iii) By Lemma \ref{Lem. transformation of Bq(k)}, Lemma \ref{Lem. eigenvalues of Cq(k)}, (\ref{Eq. definition of lambda b}) and (\ref{Eq. an easy evaluation for roots of unity}) again, and noting that $2\nmid p$, one can verify that 
	\begin{align*}
		 \det [G_p(\chi^{-2i+2j})^{-1}]_{0\le i,j\le (p-3)/2}
		&=\left(\frac{p-1}{2p}\right)^{\frac{p-1}{2}}\prod_{b\in\mathbb{F}_p^{\times}/U_2}
			\left(\zeta_p^{b}+\zeta_p^{-b}-2\right)\\
		&=\left(\frac{p-1}{2p}\right)^{\frac{p-1}{2}}\prod_{b\in\mathbb{F}_p^{\times}/U_2}
		\left(\zeta_p^{2b}+\zeta_p^{-2b}-2\right)\\
		&=\left(\frac{p-1}{2p}\right)^{\frac{p-1}{2}}\prod_{b\in\mathbb{F}_p^{\times}/U_2}
		\left(\zeta_p^{b}-\zeta_p^{-b}\right)^2\\
		&=\left(\frac{p-1}{2p}\right)^{\frac{p-1}{2}}\cdot (-1)^{\frac{p-1}{2}}\prod_{b\in\mathbb{F}_p^{\times}}
		\left(\zeta_p^{b}-\zeta_p^{-b}\right)\\
		&=\left(\frac{p-1}{2p}\right)^{\frac{p-1}{2}}\cdot (-1)^{\frac{p-1}{2}}\prod_{b\in\mathbb{F}_p^{\times}}\zeta_p^b
		\left(1-\zeta_p^{-2b}\right)\\
		&=\left(\frac{p-1}{2p}\right)^{\frac{p-1}{2}}\cdot (-1)^{\frac{p-1}{2}}\prod_{b\in\mathbb{F}_p^{\times}}
		\left(1-\zeta_p^{b}\right)\\
		&=(-1)^{\frac{p-1}{2}}p\left(\frac{p-1}{2p}\right)^{\frac{p-1}{2}}.
	\end{align*}
	Now by Lemma \ref{Lem. the Lerch permutation} we obtain 
	\begin{equation*}
		\det B_p(2)=(-1)^{\frac{(p+3)(p-1)}{4}}p\left(\frac{p-1}{2p}\right)^{\frac{p-1}{2}}.
	\end{equation*}
    
    In view of the above, we have completed the proof of Theorem \ref{Thm. Bq(1) and Bq(2)}.\qed

	\Ack\ The first author thanks Prof. Hao Pan for his steadfast encouragement. This work was supported by the Natural Science Foundation of China (Grant Nos. 12101321 and 12201291).

\end{document}